\documentclass[a4paper,11pt,titlepage,twoside]{article}

\usepackage{graphicx}
\usepackage[T1]{fontenc} 
\usepackage[utf8]{inputenc}
\usepackage[english]{babel}
\usepackage{soul}
\usepackage{amsfonts}
\usepackage{amsmath}
\usepackage{amsthm}
\usepackage{amssymb}
\usepackage{mathrsfs}
\usepackage[top=5cm, bottom=3cm, left=3.2cm, right=3.2cm]{geometry}
\usepackage{setspace}
\usepackage{afterpage}
\usepackage{extarrows}
\usepackage{fancyhdr}
\usepackage{titlesec}
\usepackage{enumitem} \setlist{nosep}
\usepackage[pdftex,breaklinks,colorlinks,linkcolor=blue,
anchorcolor=blue]{hyperref}

\theoremstyle{definition}
\newtheorem{defin}{Definition}[section]
\theoremstyle{plain}
\newtheorem{theo}[defin]{Theorem}

\newtheorem{pro}[defin]{Proposition}

\theoremstyle{definition}
\newtheorem{exm}[defin]{Example}
\newtheorem{rem}[defin]{Remark}
\newtheorem{rems}[defin]{Remarks}

\renewcommand{\H}{\mathcal{H}}

\newcommand{\nor}{\|\cdot\|}

\renewcommand{\l}{\langle}
\renewcommand{\r}{\rangle}
\newcommand{\N}{\mathbb{N}}
\newcommand{\Z}{\mathbb{Z}}

\newcommand{\R}{\mathbb{R}}

\newcommand{\pint}{\l\cdot,\cdot\r}
\newcommand{\pin}[2]{\l#1 , #2\r}

\newcommand{\mez}{\frac{1}{2}}
\renewcommand{\phi}{\varphi}
\numberwithin{equation}{section}

\renewcommand\labelenumi{\emph{(\roman{enumi})}}
\renewcommand\theenumi\labelenumi

\titleformat{\section}
{\normalfont\fillast \fontsize{12}{15}\scshape}{\thesection.}{0.8em}{}

\titleformat{\subsection}
{\normalfont\fillast \fontsize{11}{12}\scshape}{\thesubsection.}{0.8em}{}

\begin{document}
	
\thispagestyle{plain}

\begin{center}
	\large
	{\uppercase{\bf Estimate of the spectral radii of \\Bessel multipliers and consequences}} \\
	\vspace*{0.5cm}
	{\scshape{Rosario Corso}}
\end{center}

\normalsize 
\vspace*{1cm}	

\small 

\begin{minipage}{12.5cm}
	{\scshape Abstract.} Bessel multipliers are operators defined from two Bessel sequences of elements of a Hilbert space and a complex sequence, and have frame multipliers as particular cases.  In this paper an estimate of the spectral radius of a Bessel multiplier is provided involving the cross Gram operator of the two sequences. 	
	As an upshot, it is possible to individuate some regions of the complex plane where the spectrum of a multiplier of dual frames is contained.  
\end{minipage}

\vspace*{.5cm}

\begin{minipage}{12.5cm}
	{\scshape Keywords:} Bessel multipliers, dual frames, spectral radius, spectrum.
\end{minipage}

\vspace*{.5cm}

\begin{minipage}{12.5cm}
	{\scshape MSC (2010):}  42C15, 47A10, 47A12.
\end{minipage}

\normalsize

\section{Introduction}

Bessel multipliers, as introduced in \cite{Balazs_basic_mult}, are operators in Hilbert spaces which have been extensively studied \cite{Balazs_inv_mult2,Corso_distr_mult,Balazs_inv_mult,Dual_mult}, occur in various fields of applications \cite{Balazs_appl,Gazeau,Matz} and include the class of frame multipliers \cite{Corso_mult1,Corso_mult2,FeiNow,JAA,LA}. 
Recently, in \cite{Corso_mult2}, given a frame multiplier some regions of the complex plane containing the spectrum have been identified. In order to present the main contributions of this paper, which follows the line of \cite{Corso_mult2}, we need to give some definitions and preliminary results. 

A {\it Bessel sequence} of a separable Hilbert space $\H$ (with inner product $\pint$ and norm $\nor$) is a sequence $\varphi=\{\varphi_n\}_{n\in \N}$ of elements of $\H$ such that  
$$
\sum_{n\in \N} |\pin{f}{\varphi_n}|^2\leq B_\varphi \|f\|^2, \qquad \forall f \in \H
$$
for some $B_\varphi>0$ (called a {\it Bessel bound} of $\varphi$).  
A sequence $\varphi=\{\varphi_n\}_{n\in \N}$ is a {\it frame} for $\H$ if there exist $A_\varphi,B_\varphi>0$ such that
\begin{equation}
	\label{def_frame}
	A_\varphi \|f\|^2\leq \sum_{n\in \N} |\pin{f}{\varphi_n}|^2\leq B_\varphi \|f\|^2, \qquad \forall f \in \H.
\end{equation}
Given two Bessel sequences $\varphi=\{\varphi_n\}_{n\in \N},\psi=\{\psi_n\}_{n\in \N}$ of $\H$ and $m=\{m_n\}_{n\in \N}$ a bounded complex sequence (in short, $m\in \ell^\infty(\N)$) it is possible to define a bounded operator $M_{m,\varphi,\psi}$ on $\H$ 
in the following way
$$M_{m,\varphi,\psi}f=\sum_{n \in \N} m_n\pin{f}{\psi_n} \varphi_n \qquad f\in \H.$$ 
This operator is said the {\it Bessel multiplier} of $\varphi$, $\psi$ with {\it symbol} $m$. It thus consists of three processes: analysis through the sequence $\psi$, multiplication of the analysis coefficients by $m$ and synthesis processes by $\varphi$. When $\varphi$ and $\psi$ are frames, $M_{m,\varphi,\psi}$ is called a {\it frame multiplier}. 

Since a Bessel multiplier is a bounded operator, its spectrum is contained in some disk and, more precisely, the following bound has been given. 

 \begin{pro}[{\cite[Proposition 1]{Corso_mult2}}]
	\label{spec_intro}
	The spectrum $\sigma(M_{m,\varphi,\psi})$ of any Bessel multiplier $M_{m,\varphi,\psi}$ is contained in the closed disk centered the origin with radius $\sup_{n} |m_n| {B_\varphi}^\mez {B_\psi}^\mez,$ where $B_\varphi$ and $B_\psi$ are Bessel bounds of $\varphi$ and $\psi$, respectively. 
\end{pro}

A special case occurs when $\varphi$ and $\psi$ are {\it dual frames}, i.e. two frames satisfying the condition\footnote{or, equivalently, the condition $\displaystyle f=\sum_{n\in \N} \pin{f}{\phi_n}\psi_n$ for every $f\in \H$.}
\begin{equation}
	\label{dual_frames}
	f=\sum_{n\in \N} \pin{f}{\psi_n}\varphi_n, \qquad \forall f\in \H.
\end{equation}
In this setting it was possible to find more precise regions where the spectra are contained, as stated in the following result. 

\begin{pro}[{\cite[Propositions 2 and 3]{Corso_mult2}}]
	\label{pertII}
	Let $\varphi,\psi$ be dual frames for $\H$ with upper bounds $B_\varphi,B_\psi$, respectively, and let $m\in \ell^\infty(\N)$. 
	\begin{enumerate}
		\item If $m$ is contained in the disk of center $\mu$ with radius $R$, then $\sigma(M_{m,\varphi,\psi})$ is contained in the disk of center $\mu$ with radius $R{B_\varphi}^\mez {B_\psi}^\mez$.
		\item If $m$ is a real sequence, then $\sigma(M_{m,\varphi,\psi})$ is contained in the disk of center $$\frac 12 (\sup_{n} m_n +\inf_{n} m_m)$$ with radius $$\frac 12 (\sup_{n} m_n -\inf_{n} m_m){B_\varphi}^\mez {B_\psi}^\mez.$$
		\item If $\psi$ is the canonical dual\footnote{among all the dual frames of $\phi$ there is a special one called the canonical dual; the definition is given in Section \ref{sec:pre}.} of $\varphi$, then $\sigma(M_{m,\varphi,\psi})$ is contained in the closed convex hull of $m$.  
	\end{enumerate}
\end{pro}

One of the two main results of this paper, which is right below, gives an estimate of the spectral radius of a Bessel multiplier in terms of the cross Gram operator $G_{\varphi,\psi}$ \cite{Balazs08} of $\varphi$ and $\psi$ which is recalled in Section \ref{sec:pre}. A direct consequence is an improvement of Proposition \ref{spec_intro}. 

\begin{theo}
	\label{th_main}
	Let $\varphi,\psi$ be Bessel sequences of $\H$ with cross Gram operator $G_{\varphi,\psi}$ and let $m\in \ell^\infty(\N)$. Let $M_m$ be the multiplication operator by $m$ on $\ell^2(\N)$. Then $M_{m,\varphi,\psi}$ and $M_mG_{\varphi,\psi}$ have the same spectral radius
	\begin{equation}
		\label{bound_radius0}
		r(M_{m,\varphi,\psi})=  r(M_mG_{\varphi,\psi}).
	\end{equation}
	In particular, the following bound holds 
	\begin{equation}
		\label{bound_radius}
		r(M_{m,\varphi,\psi})\leq  \sup_n |m_n| \;\! \|G_{\varphi,\psi}\|.
	\end{equation}
	Therefore, the spectrum of any Bessel multiplier $M_{m,\varphi,\psi}$ is contained in the closed disk centered the origin with radius $ \sup_n |m_n| \;\! \|G_{\varphi,\psi}\|$. 
\end{theo}

Theorem \ref{th_frames}, concerning dual frames, is instead the counterpart of Proposition \ref{pertII} which involve the cross Gram operator. Both in Theorems \ref{th_main} and \ref{th_frames} the constant ${B_\varphi}^\mez {B_\psi}^\mez$ present in Propositions \ref{spec_intro} and \ref{pertII} is substituted by the norm $\|G_{\varphi,\psi}\|$ of $G_{\varphi,\psi}$. Since the inequality $\|G_{\varphi,\psi}\|\leq {B_\varphi}^\mez {B_\psi}^\mez$ always holds, Theorems \ref{th_main} and \ref{th_frames} improve in fact Propositions \ref{spec_intro} and \ref{pertII}. In connections to the main results, throughout the paper we will discuss some remarks and examples.

\section{Preliminaries}
\label{sec:pre}

We denote by $\ell^2(\N)$ (respectively, $\ell^\infty(\N)$)  the usual spaces of square summable (respectively, bounded) complex sequences indexed by $\N$. \\
Given two Bessel sequences $\varphi$ and $\psi$ of $\H$ the following operators can be defined (see \cite{Balazs08,Chris}):
\begin{itemize}
	\item $C_\varphi:\H\to \ell^2(\N)$, defined by $C_\varphi f=\{\pin{f}{\varphi_n}\}$, is the {\it analysis operator} of $\varphi$.
	\item $D_\varphi:\ell^2(\N) \to \H$, defined by $D_\varphi \{c_n\}=\sum_{n\in \N} c_n \varphi_n$, is the {\it synthesis operator} of $\varphi$.
	\item $S_\varphi:\H \to \H$, $S_\varphi=D_\varphi C_\varphi $ is called the {\it frame operator} of $\varphi$; the action of $S_\varphi$ is 
	$$
	S_\varphi f =\sum_{n\in \N} \pin{f}{\varphi_n}\varphi_n \qquad f\in \H. 
	$$
	\item $G_{\varphi,\psi}:\ell^2(\N)\to \ell^2(\N)$, $G_{\varphi,\psi}=C_\psi D_\varphi$, is the {\it cross Gram operator} of $\varphi$ and $\psi$ which acts as $G_{\varphi,\psi}\{c_n\}=\{d_k\}$, where $d_k=\sum_{n\in \N} c_n\pin{ \varphi_n}{\psi_k}$. In other words, $G_{\varphi,\psi}$ can be associated to the matrix $(\pin{ \varphi_n}{\psi_k})_{n,k\in \N}$.
\end{itemize}
Moreover, $C_\varphi$ and $D_\varphi$ are one the adjoint of the other one, $C_\varphi=D_\varphi^*$, and $\|C_\varphi\|=\|D_\varphi\|\leq{B_\varphi}^\mez$ where $B_\varphi$ is a Bessel bound of $\varphi$. Consequently,  $S_\phi$ is a positive self-adjoint operator and it is also invertible with bounded inverse $S_\phi^{-1}$ on $\H$. 
We recall that in the introduction we gave the definition of dual frames. A frame $\phi$ always has a dual frame, namely the sequence $\{S_\varphi^{-1}\varphi_n\}_{n\in \N}$, which is the so-called {\it canonical dual} of $\phi$. 

Finally, we note that, introducing the operators $D_\varphi$ and $C_\psi$, it is possible to write $M_{m,\varphi,\psi}=D_\varphi M_m C_\psi$ where $M_m$ is the multiplication operator by $m$ on $\ell^2(\N)$, defined by $M_m\{c_n\}=\{m_nc_n\}$ for $\{c_n\}\in \ell^2(\N)$.

\section{Proofs of the main results}
\label{sec_bas}

Theorem \ref{th_main} concerns the spectral radius of a Bessel multiplier. For a bounded operator $T:\H\to \H$, we write $\sigma(T)$ for the spectrum and $r(T):=\sup\{|\lambda|:\lambda \in \sigma(T)\}$ for the {\it spectral radius} (see, for instance, \cite{Conway,Kato,Schm}). The spectral radius represents then the radius of the smallest disk centered in the origin and containing the spectrum. Propositions \ref{spec_intro} and \ref{pertII} can be restated in terms of spectral radius. For example, we can say that for any Bessel multiplier $M_{m,\varphi,\psi}$ we have $r(M_{m,\varphi,\psi})\leq \sup_{n} |m_n| {B_\varphi}^\mez {B_\psi}^\mez$.

For the proof of Theorem \ref{th_main} below we are going to use some classical results about the spectral radius (see e.g. \cite[Proposition 3.8]{Conway}): for every bounded operator $T:\H \to \H$ we have 
\begin{equation}
	\label{sp_rad1}
r(T)=\lim_{N\to +\infty} \|T^N\|^{\frac 1 N},
\end{equation} 
and 
\begin{equation}
	\label{sp_rad2}
	r(T)\leq \|T\|.
\end{equation}

\begin{proof}[Proof of Theorem \ref{th_main}]
	If $B_\varphi=0$, $B_\psi=0$  or $m\equiv 0$, then \eqref{bound_radius} trivially holds, because both the operators $M_{m,\varphi,\psi}$ and $M_mG_{\varphi,\psi}$ are null. So we can assume that $B_\varphi,B_\psi>0$ and that $m$ is not identically null. \\
	Since $M_{m,\varphi,\psi}=D_\varphi M_m C_\psi$ and $G_{\varphi,\psi}=C_\psi D_\varphi$, then for $N\geq 2$ we have 
	\begin{align*}
		M_{m,\varphi,\psi}^N&=(D_\varphi M_m C_\psi)^N=D_\varphi  (M_m C_\psi D_\varphi)^{N-1} M_m C_\psi =D_\varphi  (M_m G_{\varphi,\psi})^{N-1} M_m C_\psi.
	\end{align*}
	Therefore, 
	$$
		\|M_{m,\varphi,\psi}^N\|\leq \|(M_m G_{\varphi,\psi})^{N-1}\| \|M_m\|\|C_\psi\|\|D_\varphi\|.
	$$
	Thus, by \eqref{sp_rad1}, 
	\begin{align*}
		r(M_{m,\varphi,\psi})&=\lim_{N \to +\infty} \|M_{m,\varphi,\psi}^N\|^{\frac 1 N}\leq \lim_{N \to +\infty} (  \|(M_m G_{\varphi,\psi})^{N-1}\| \|M_m\|\|C_\psi\|\|D_\varphi\|)^{\frac 1 N}\\
		&= \lim_{N \to +\infty}   \|(M_m G_{\varphi,\psi})^{N-1}\|^{\frac 1 N} \lim_{N \to +\infty} (\|M_m\|\|C_\psi\|\|D_\varphi\|)^{\frac 1 N}=r(M_m G_{\varphi,\psi}). 
	\end{align*}
For the reverse inequality we observe that 
\begin{align*}
	(M_m G_{\varphi,\psi})^{N+1}&= M_m C_\psi D_\varphi  (M_m G_{\varphi,\psi})^{N-1} M_m C_\psi D_\varphi= M_m C_\psi M_{m,\varphi,\psi}^ND_\varphi.
\end{align*}
Hence, with an analog calculation as before we find that
$	r(M_m G_{\varphi,\psi})\leq r(M_{m,\varphi,\psi})$, so in conclusion \eqref{bound_radius0} is proved. 
Lastly, \eqref{bound_radius} holds because by \eqref{bound_radius0} and \eqref{sp_rad2} we have 
\[
r(M_{m,\varphi,\psi})=r(M_m G_{\varphi,\psi})\leq \|M_m G_{\varphi,\psi}\|\leq \|M_m\|\| G_{\varphi,\psi}\|=\sup_n |m_n| \;\! \|G_{\varphi,\psi}\|. \qedhere\]
\end{proof}

\begin{rems}
	\label{rems}
	\begin{enumerate}
		\item[(i)] Inequality \eqref{bound_radius} may be strict. In fact, let $\{e_n\}$ be an orthonormal basis of $\H$, $\varphi=\{e_n\}$, $\psi=\{\frac 12e_1,2e_2,\frac 12e_3,2e_4, \dots\}$ and $m=\{2,\frac 12,2,\frac 12, \dots \}$. A trivial calculation shows that $M_{m,\varphi,\psi}$ is the identity operator, so $r(M_{m,\varphi,\psi})=1$, while $\sup_n |m_n| \;\! \|G_{\varphi,\psi}\|=4$. 
		\item[(ii)] A Riesz basis $\phi$ for $\H$ is the image of an orthonormal basis $\{e_n\}$ of $\H$ through an bounded operator with bounded inverse defined on $\H$ \cite{Chris}. A Riesz basis $\phi$ is, in particular, a frame for $\H$ and it has a unique dual $\psi$ (the canonical one) which is a Riesz basis too. Moreover,  $$\pin{\phi_n}{\psi_k}=\delta_{n,k}=\begin{cases}
			1  \qquad n=k\\
			0  \qquad n\neq k.
		\end{cases}$$
		Therefore, if $\phi$ is a Riesz basis for $\H$ and $\psi$ is its canonical dual, then $G_{\varphi,\psi}$ is the identity operator on $\ell^2(\N)$ and so  $\|G_{\varphi,\psi}\|=1$.  Anyway, for this choice of $\phi,\psi$, \eqref{bound_radius} is an immediate consequence of the fact that $\sigma(M_{m,\varphi,\psi})$ is the closure of $\{m_n:n\in \N\}$ (see \cite[Proposition 4]{Corso_mult1}). 
		\item[(iii)] Since $G_{\varphi,\psi}=C_\psi D_\varphi$, we always have  
		\begin{equation}
			\label{st_G}
			\|G_{\varphi,\psi}\|\leq \|C_\psi\|\| D_\varphi\|\leq {B_\varphi}^\mez {B_\psi}^\mez.
		\end{equation} 
		Therefore, Theorem \ref{th_main} is finer than Proposition \ref{spec_intro}. 
		Moreover, if $\phi=\psi$, then $G_{\varphi,\varphi}=C_\varphi D_\varphi=D_\varphi^* D_\varphi$ is a positive self-adjoint operator, so $\|G_{\varphi,\varphi}\|=\|D_\varphi\|^2$. 
	\end{enumerate}
\end{rems}		

Besides with \eqref{st_G} it is possible to estimate the norm of $G_{\varphi,\psi}$ with some other considerations which we discuss below. 
\begin{rems}
	\begin{enumerate}
		\item[(i)] An estimate of $\|G_{\varphi,\psi}\|$ can be given if 
		\begin{equation}
			\label{rem_schur}
			\displaystyle \sup_{k\in \N} \sum_{n\in \N} |\pin{\varphi_n}{\psi_k}|\leq \Gamma_1 \;\text{ and }\; \displaystyle \sup_{n\in \N} \sum_{k\in \N} |\pin{\varphi_n}{\psi_k}|\leq \Gamma_2.
		\end{equation} Indeed, by Schur test (see for instance \cite[Lemma 6.2.1]{Groechenig_b}), we have $\|G_{\varphi,\psi}\|\leq \Gamma_1^\mez \Gamma_2^\mez$.
		\item[(ii)] Let $\phi,\psi$ be Bessel sequences of $\H$ such that  $\pin{\varphi_n}{\psi_k}=0$ for $n,k\in \N$ with $|n-k|>d$.  As particular case of the previous remark, if $\displaystyle \sum_{i=-d}^{d} \sup_{n}|\pin{\varphi_n}{\psi_{n+i}}|\leq \Gamma$
		(where, with an abuse of notation, we mean $\psi_{-d+1},\dots, \psi_{-1}, \psi_{0}=0$) then $\|G_{\varphi,\psi}\|\leq  \Gamma$.
		\item[(iii)] Another use of conditions \eqref{rem_schur} can be made in the context of localized frames \cite{BCHL,CG,Groechenig_l,GF}.
	\end{enumerate}
\end{rems}

In what follows we give another example where in particular it is possible to exactly calculate the norm of the cross Gram operator. 

\begin{exm}
	Let $\mathcal G$ be a countable locally compact abelian group equipped with the discrete topology. We write the group operation of $\mathcal G$ in the additive notation and we denote by $\widehat{\mathcal G}$ the dual group of $\mathcal G$ (i.e. the multiplicative group of the characters on $\mathcal G$). Since $\mathcal G$ is discrete, then $\widehat{\mathcal G}$ is compact (see \cite[Proposition 4.4]{Folland}). Moreover, we will choose the Haar measure on $\mathcal G$ to be the counting measure; hence by \cite[Proposition 4.24]{Folland}, $|\widehat{\mathcal G}|=1$. 
	
	Let $\tau$ be a unitary representation of $\mathcal G$ on $\H$. In particular, let us assume that $\tau$ is dual integrable \cite{HSWW}, i.e. there exist a Haar measure $d\xi$ and a function $[\cdot, \cdot]:\H \times \H \to L^1(\widehat{\mathcal G},d\xi)$ such that 
	\begin{equation}
		\pin{\chi}{\tau_g \eta}=\int_{\widehat{\mathcal G}} [\chi,\eta](\xi)e_{-g}(\xi) d\xi \qquad \forall g\in \mathcal G, \chi, \eta \in \H,
	\end{equation}
	where $e_{-g}(x)$ is the character induced by $-g$, namely $e_{-g}(\xi)=e^{-2\pi i (g,\xi)}$, and $(\cdot,\cdot)$ is the duality between $\mathcal G$ and $\widehat{\mathcal G}$. 
	The function $[\cdot, \cdot]$ is called the {\it bracket function}.  
	Classical examples (treated for instance in \cite{Chris,Groechenig_b}) of this framework are
	\begin{itemize}
		\item $\mathcal G=\Z^d$, $\widehat{\mathcal G}=\mathbb{T}^n$, $\H=L^2(\R)$, $(\tau_k f)(x)=(T_k f)(x)=f(x-k)$ for $k\in \Z^d$ and 
		$$
		[\chi,\eta](\xi)=\sum_{k\in \Z^d} \hat \chi (\xi+k)\overline{\hat \eta (\xi+k)}, \qquad \xi \in \R^d, \chi,\eta \in L^2(\R^d),
		$$
		where $\hat \chi$ and $\hat \eta$ are the Fourier transforms of $\chi$ and $\eta$, respectively;
		\item $\mathcal G=\Z^d\times \Z^d$, $\widehat{\mathcal G}=\mathbb{T}^n$, $\H=L^2(\R)$, $(\tau_{(k,l)} f)(x)=(T_k M_l f)(x)=e^{2\pi i l \cdot x}f(x-k)$ for $(k,l)\in \Z^d\times \Z^d$ and 
		$$
		[\chi,\eta](x,\xi)=Z\chi(x,\xi)\overline{Z\eta(x,\xi)}, \qquad x,\xi \in \R^d, \chi,\eta \in L^2(\R^d),
		$$
		where $Z\chi(x,\xi)=\sum_{k\in \Z^d} e^{-2\pi i k \cdot \xi}\chi(x-k)$ is the Zak transform of $\chi\in L^2(\R^d)$. 
	\end{itemize}
	After introducing this setting, we now consider two special sequences\footnote{In this example, the sequences are indexed by the countable set $\mathcal G$ in contrast to the setting of the rest of the paper. However, this does not change the validity of Theorems \ref{th_main} and \ref{th_frames} since the series defining a multiplier is unconditionally convergent so the ordering of a Bessel sequence is not relevant (see \cite{Chris,Groechenig_b}).}. More precisely, let $\chi,\eta\in \H$ be such that $\varphi=\{T_g \chi \}_{g\in \mathcal G}$ and $\psi=\{T_g \eta \}_{g\in \mathcal G}$ are Bessel sequences\footnote{This happen if and only if $[\chi,\chi]$ and $[\eta,\eta]$ are bounded above a.e. in $\widehat{\mathcal G}$, see \cite[Section 5]{HSWW}.} of $\H$. As we are going to see, the norm $\|G_{\varphi,\psi}\|$ can be exactly calculated in terms of $[\chi,\eta]$. Indeed, for any complex sequences $\{c_g\}_{g\in \mathcal G},\{d_g\}_{g\in \mathcal G}\in \ell^2(\mathcal G)$ we have 
	\begin{equation}
		\label{G_mult}
		\begin{aligned}
			\pin{G_{\varphi,\psi}\{c_g\}}{\{d_g\}}&=\pin{C_{\psi}D_{\varphi}\{c_g\}}{\{d_g\}}=\pin{D_{\varphi}\{c_g\}}{D_{\psi}\{d_g\}}\\
			&=\left \l \sum_{g\in \mathcal G} c_gT_g \chi,\sum_{g\in \mathcal G} d_gT_g \eta \right \r= \sum_{g,h\in \mathcal G} c_g\overline{d_h} \pin{T_g \chi}{T_h \eta}\\
			&=\sum_{g,h\in \mathcal G} c_g\overline{d_h} \pin{\chi}{T_{h-g} \eta}= \sum_{g,h\in \mathcal G} c_g\overline{d_h} \int_{\widehat{\mathcal G}} [\chi,\eta](\xi)e_{g-h}(\xi) d\xi\\
			&=\int_{\widehat{\mathcal G}} [\chi,\eta](\xi)\sum_{g,h\in \mathcal G} c_g\overline{d_h} e_{g-h} (\xi) d\xi\\
			&=\int_{\widehat{\mathcal G}} [\chi,\eta](\xi)\sum_{g\in \mathcal G} c_g e_g(\xi) \overline{\sum_{h\in \mathcal G}d_h e_{h} (\xi)} d\xi.
		\end{aligned}
	\end{equation}
	By the Pontrjagin duality theorem and by \cite[Corollary 4.26]{Folland}, $\{e_g\}_{g\in \mathcal G}$ is an orthonormal basis of $L^2(\widehat{\mathcal G},d\xi)$. 
	This fact, together with \eqref{G_mult}, implies that the Gram operator $G_{\varphi,\psi}$ can be reduced to the multiplication operator by $[\chi,\eta]$ on $L^2(\widehat{\mathcal G},d\xi)$. Hence, we conclude that 
	$$\|G_{\varphi,\psi}\|=\sup_{\{c_g\},\{d_g\}\neq 0}\frac{|\pin{G_{\varphi,\psi}\{c_g\}}{\{d_g\}}|}{\|\{c_g\}\|\|\{d_g\}\|}= \sup_{\xi\in \widehat{\mathcal G}} |[\chi,\eta](\xi)|,
	$$
	i.e. the essential supremum of $[\chi,\eta]$ (see \cite[Example 2.11 - Ch. III]{Kato}). Thus, by Theorem \ref{th_main}, given a bounded complex sequence $m=\{m_g\}_{g\in \mathcal G}$ we have 
	\begin{equation*}
		r(M_{m,\varphi,\psi})\leq  \sup_{g\in \mathcal G} |m_g|  \sup_{\xi\in \widehat{\mathcal G}} |[\chi,\eta](\xi)|.
	\end{equation*}
\end{exm}

We now move to prove the result for dual frames. In particular, it provides regions containing the spectrum which are smaller than the disk of Theorem \ref{th_main}.

\begin{theo}
	\label{th_frames}
	Let $\varphi,\psi$ be dual frames for $\H$ and let $m\in \ell^\infty(\N)$. 
	\begin{enumerate}
		\item If $m$ is contained in the disk of center $\mu$ with radius $R$, then $\sigma(M_{m,\varphi,\psi})$ is contained in the disk of center $\mu$ with radius $R\|G_{\varphi,\psi}\|$.
		\item If $m$ is real, then $\sigma(M_{m,\varphi,\psi})$ is contained in the disk of center $$\frac 12 (\sup_{n} m_n +\inf_{n} m_m)$$ with radius $$\frac 12 (\sup_{n} m_n -\inf_{n} m_m)\|G_{\varphi,\psi}\|.$$
		\item If $\psi$ is the canonical dual of $\varphi$, then $\sigma(M_{m,\varphi,\psi})$ is contained in the closed convex hull of $m$.  
	\end{enumerate} 
\end{theo}

\begin{proof}
To prove statement (i), let us consider a disk of center $\mu$ with radius $R$ containing the sequence $m$. By \eqref{dual_frames} we have 
$$
M_{m,\varphi,\psi} -\mu I=\sum_{n\in \N} (m_n-\mu)\pin{f}{\psi_n} \varphi_n=M_{m-\mu,\varphi,\psi},
$$
where with $m-\mu$ we mean the sequence $\{m_n-\mu\}$. Therefore applying \eqref{bound_radius} to $M_{m-\mu,\varphi,\psi}$, we obtain 
$$r(M_{m,\varphi,\psi}-\mu I)\leq  \sup_n |m_n-\mu| \;\! \|G_{\varphi,\psi}\|\leq R \|G_{\varphi,\psi}\|,$$
which means that $\sigma(M_{m,\varphi,\psi})$ is contained in the disk of center $\mu$ with radius $R\|G_{\varphi,\psi}\|$, because $\sigma(M_{m,\varphi,\psi})=\{\lambda+\mu : \lambda \in \sigma(M_{m-\mu,\varphi,\psi})\}$. 

Statement (ii) is a consequence of (i) taking $\mu=\frac 12 (\sup_{n} m_n +\inf_{n} m_m)$ and $R=\frac 12 (\sup_{n} m_n -\inf_{n} m_m)$. Finally, statement (iii) was proved in \cite[Proposition 2]{Corso_mult2}. 
\end{proof}

By \eqref{st_G} we can make a similar observation of Remark \ref{rems}(iii), that is Theorem \ref{th_frames} is stronger than Proposition \ref{pertII}. We conclude with a comment for the case of a frame and its canonical dual.

\begin{rem}
	Let $\phi$ and $\psi$ be dual frames. Making use of inequality \eqref{bound_radius} (taking $m_n=1$ for every $n\in \N$), we find that $\|G_{\varphi,\psi}\|\geq 1$. \\If, in particular, $\psi$ is the canonical dual of $\phi$, then $$\pin{\varphi_n}{\psi_k}=\pin{\varphi_n}{S_\phi^{-1} \varphi_k}=\pin{S_\phi^{-\mez}  \varphi_n}{S_\phi^{-\mez} \varphi_k}.$$ In other words, $G_{\varphi,\psi}$ is equals to the Gram operator of the {\it canonical tight frame} $\chi:=S_\phi^{-\mez}\varphi$ of $\varphi$, which is a Parseval frame (i.e. it satisfies condition \eqref{def_frame} with $A_\chi=B_\chi=1$, see \cite[Theorem 6.1.1]{Chris}). Thus, for the initial observation and for Remark \ref{rems}(iii), if $\phi$ is a frame, $\psi$ is its canonical dual and $\chi=S_\phi^{-\mez}\varphi$, then we have $1\leq \|G_{\varphi,\psi}\|=\|D_{\chi}\|^2\leq B_\chi=1$, so $\|G_{\varphi,\psi}\|=1$. 
\end{rem}

\section*{Acknowledgments}

This work was partially supported by the European Union through the Italian Ministry of University and Research (FSE - REACT EU, PON Ricerca e Innovazione 2014-2020) and by the ``Gruppo Nazionale per l'Analisi Matematica, la Probabilità e le loro Applicazioni'' (INdAM).  This work has been done within the activities of the ``Gruppo UMI - Teoria dell’Approssimazione e Applicazioni''.

\vspace*{0.5cm}
\begin{center}
\textsc{Rosario Corso, Dipartimento di Matematica e Informatica} \\
\textsc{Università degli Studi di Palermo, I-90123 Palermo, Italy} \\
{\it E-mail address}: {\bf rosario.corso02@unipa.it}
\end{center}

\end{document}